\documentclass[preprint,12pt]{elsarticle}

\usepackage{color}

\usepackage[latin1]{inputenc}
\usepackage{amsmath,amsthm,amssymb}
\usepackage{amsfonts}
\usepackage{amsmath,amsthm,amssymb,amscd}
\usepackage{latexsym}
\usepackage{color}
\usepackage{graphicx}
\usepackage{mathrsfs}

\makeatletter \@addtoreset{equation}{section} \makeatother

\newtheorem{theorem}{Theorem}[section]

\newtheorem{proposition}{Proposition}[section]
\newtheorem{lemma}{Lemma}[section]
\newtheorem{remark}{Remark}[section]

\newtheorem{corollary}[theorem]{Corollary}
\newtheorem{teo}{Theorem}[section]

\newcommand{\R}{\mathbb{R}}
\newcommand{\N}{\mathbb{N}}

\newcommand{\norm}[1]{\|#1\|}

\begin{document}

\begin{frontmatter}

\title{On isolated singularities  of  Kirchhoff--type Laplacian problems}

\author[HC]{Huyuan Chen}
\address[HC]{Department of Mathematics, Jiangxi Normal University,\\
Nanchang, Jiangxi 330022, P.R. China}
\ead{chenhuyuan@yeah.net}

\author[MF]{Mouhamed Moustapha Fall}
\address[MF]{African Institute for Mathematical Sciences in  Senegal\\
KM2, Rt de Joal, Mbour. BP 1418. Senegal}
\ead{mouhamed.m.fall@aims-senegal.org}

\author[BZ]{Binlin Zhang }
\address[BZ]{Department of Mathematics, Heilongjiang Institute of Technology,\\
Harbin 150050, P.R. China}
\ead{zhangbinlin2012@163.com}
\cortext[cor1]{Corresponding author}

\begin{abstract}

In this paper, we study isolated singular positive solutions for the following Kirchhoff--type Laplacian problem:
\begin{equation*}
-\left(\theta+\int_{\Omega} |\nabla u| dx\right)\Delta u =u^p   \quad{\rm in}\quad  \Omega\setminus \{0\},\qquad  u=0\quad  {\rm on}\quad  \partial \Omega,
\end{equation*}
where  $p>1$, $\theta\in \R$, $\Omega$ is a bounded smooth domain containing the origin
 in $\R^N$ with $N\ge 2$.
In the subcritical case: $1<p<N/(N-2)$ if $N\ge3$, $1<p<+\infty$ if $N=2$,  we employ the Schauder fixed-point theorem to derive  a sequence of positive isolated singular solutions for the above problem such that $M_\theta(u)>0$. To estimate  $M_\theta(u)$, we make use of the rearrangement argument.
Furthermore, we  obtain a sequence of isolated singular  solutions such that $M_\theta(u)<0$,  by analyzing relationship between the parameter $\lambda$ and the unique solution $u_\lambda$ of $$-\Delta u+\lambda u^p=k\delta_0\quad{\rm in}\quad B_1(0),\qquad  u=0\quad  {\rm on}\quad  \partial  B_1(0).$$
In the supercritical case: $N/(N-2)\le p<(N+2)/(N-2)$ with $N\ge3$, we obtain two isolated singular  solutions $u_i$ with $i=1,2$ such that $M_\theta(u_i)>0$
under some appropriate assumptions.
\end{abstract}


\end{frontmatter}
\section{Introduction and main results}
A model with small variation of tension due to the changes of the length of a string  is described by  D'Alembert
wave equation, it is also well-known as  the Kirchhoff equation, see \cite{K}, which states as follows
$$m \frac{\partial^2 u}{\partial t^2}-  \left[\tau_0+\frac{\kappa}{2L_0}\int_\alpha^\beta \left|\frac{\partial u}{\partial x}\right|^2 dx\right] \frac{\partial^2 u}{\partial x^2}=0,$$
where $\tau_0$ is the tension, $L_0=\beta-\alpha$ is the length of the string at rest, $m$ is the mass density, $\kappa$ is the Young's modulus. The Kirchhoff--type problems have been
attracted great attentions in the analysis of different nonlinear term due to the gradient term, see \cite{DS,FV, M,VG}.

Observe that in the prototype of Kirchhoff model, the tension, for small deformations
of the string, takes the linear form as follows:
\begin{equation}\label{1.1}
 M(u)=a+b\int_{\Omega}\sqrt{1+ |\nabla u|^2}\, dx,
\end{equation}
where $a>0,\ b>0$. When the displacement gradient is small, i.e. $|\nabla u|\ll 1$, $M(u)\backsim a+b|\Omega|+\frac b2 \int_{\Omega} |\nabla u|^2dx$.  The advantage for this approximation makes the problem have
variational structure and  the approximating solution could be constructed by variational methods.
For example, the stationary analogue   and   qualitative properties of  solutions  to the Kirchhoff--type equation
$$-\left(a+b\int_{\Omega} |\nabla u|^2dx\right)\Delta u+V(x)u=f(x,u)\quad {\rm in}\quad \Omega$$
has been extensively studied in \cite{DS,DP,HL,LPWX,P,PZ} and   extended into the fractional setting in \cite{ MRS, MS, PXZ} and the references therein.
In this case, $M(u)= a+b\int_{\Omega} |\nabla u|^2dx$ is often called Kirchhoff function.
In fact, the Kirchhoff function has been greatly extended for recent years.
For example, the case $a=0, b>0$, which is called degenerate, has been intensely investigated recently,
we refer to \cite{PXZ2} for a physical explanation and \cite{DAS, XZZ} for related results in this direction.

Our interest of this paper is to study a new Kirchhoff--type problem by taking into account that $|\nabla u|$ is not small in a bounded smooth domain $\Omega$ and the tension could be vector in a proper coordinate axis. In this situation, the Kirchhoff function \eqref{1.1} may be
taken as
\begin{equation}\label{kirchhoff}
 M_\theta(u)=\theta+\int_{\Omega} |\nabla u| dx,
\end{equation}
where $\theta$ is assumed to be real number. Given a sequence of extra pressures $\{\sigma_m\}$ with the support in $B_{\frac1m}(0)$ and the total force
$F=\int_{\Omega}\sigma_m dx=1$ keeps invariant. The limit of $\{\sigma_m\}$ as $m\to+\infty$ in the distributional sense is Dirac mass. As we know that the corresponding solutions may blow up at the origin or  blow up in the whole domain. Our aim is to
clarify this limit phenomena of the solutions to some elliptic problems involving the Kirchhoff--type function \eqref{kirchhoff}.

More precisely, in this article we are interested in nonnegative singular  solutions of the following Kirchhoff--type equation:
\begin{equation}\label{eq 1.1}
 \arraycolsep=1pt\left\{
\begin{array}{lll}
 \displaystyle  - M_\theta(u) \Delta u =u^p\quad
 & {\rm in}\quad \Omega\setminus \{0\},\\[2mm]
 \phantom{    M (u)  -\Delta   }
    u=0\quad
  &{\rm on}\quad\partial \Omega,
\end{array}\right.
\end{equation}
where $p>1$,  $M_\theta$ is defined by (\ref{kirchhoff}) with $\theta\in \R$ and $\Omega$ is a bounded, smooth  domain containing the origin in $\R^N$ with $N\ge2$.
The following parameter plays an important role in obtaining the solutions of (\ref{eq 1.1}):
\begin{equation}\label{1.1-2}
  a_p=\sup_{x\in\Omega} \frac{w_1}{w_0},
\end{equation}
 where $w_0=\mathbb{G}_\Omega[\delta_0]$ and $w_1=\mathbb{G}_\Omega[w_0^p]$, $\mathbb{G}_\Omega$ is Green operator defined
as $$\mathbb{G}_\Omega[u](x)=\int_{\Omega} G_\Omega (x,y)u(y) dy,$$
here $G_\Omega$ is the Green kernel of $-\Delta$ in $\Omega\times\Omega$ with zero Dirichlet boundary condition.  Note that
$a_p$ is well-defined when $p$ is subcritical, that is, $p<p^*$, where
 \begin{equation}\label{p*}
 p^*=  \left\{
\begin{array}{lll}
 \displaystyle \frac{N}{N-2}\quad
 & {\rm if}\quad N\ge3,\\[2mm]
 \phantom{    }
 \displaystyle +\infty\quad &{\rm if}\quad N=2.
\end{array}\right.
 \end{equation}

Our first existence result about isolated singular solutions with $M_\theta(u)>0$  is stated as follows.

\begin{teo}\label{teo 0}
 Assume that $N\ge2$, $M_\theta$ is defined by (\ref{kirchhoff}) with $\theta\in \R$, $a_p$ is given by (\ref{1.1-2}), $p^*$ is given by (\ref{p*})
and   $\Omega$ is a bounded smooth domain containing the origin such that
$$  B_1(0)\subset   \Omega\quad{\rm and}\quad |B_{r_0}(0)|=|\Omega|$$
where $1\le r_0<+\infty.$

Let   $k>r_0\theta_- $  with $\theta_-:=\min\{0,\theta\}$ be such that
  \begin{equation}\label{1.1-1}
 \frac{k^{p-1}}{\theta+r_0^{-1}k}\le \frac1{a_p p }\left(\frac{p-1}{p}\right)^{p-1}.
  \end{equation}
Then for $p\in (1,p^*)$, problem (\ref{eq 1.1}) has a nonnegative solution $u_k$ satisfying that
\begin{equation}\label{md}
 M_\theta(u_k)\ge\theta+r_0^{-1}k>0
\end{equation}
and $u_k$ has following asymptotic behaviors at the origin
\begin{equation}\label{1.2}
\lim_{|x|\to0^+} u_k(x)\Phi^{-1} (x)= c_N    k,
\end{equation}
 where $c_N>0$ is the normalized constant and
  $$\Phi (x)=\arraycolsep=1pt\left\{
\begin{array}{lll}
 \displaystyle  |x|^{2-N}\quad
 & {\rm if}\quad N\ge3,\\[2mm]
 \phantom{    }
 \displaystyle -\ln |x|\quad &{\rm if}\quad N=2.
\end{array}\right.
$$

Furthermore, $u_k$ is a distributional solution of
 \begin{equation}\label{eq 1.2}
    \arraycolsep=1pt\left\{
\begin{array}{lll}
 \displaystyle   -\Delta u =\frac{u^p}{ M_\theta(u)}+ k\delta_0\quad
 &{\rm in}\quad \Omega,\\[2mm]
 \phantom{   -\Delta   }
 \displaystyle   u=0\quad
 &{\rm on}\quad\partial \Omega,
\end{array}\right.
\end{equation}
where $\delta_0$ is Dirac mass concentrated at the origin.

\end{teo}
\begin{remark}\label{re 1}
{\rm Note that $a_p$ depends on $p$ and $\Omega$ and the value $p=2$ is critical for  assumption (\ref{1.1-1}) for $N=2,3$. Indeed,  $p^*>2$  occurs only for $N=2$ and $N=3$.
Due to the parameter $\theta$, (\ref{1.1-1}) gives a rich structure of isolated singular solutions for problem (\ref{eq 1.1}). Moreover, a discussion is put  in Proposition \ref{pr 2.1} in Section 2. }
\end{remark}

Involving Kirhchoff function $M_\theta(u)$, the classical method of Lions' iteration argument  in \cite{Lions1980} does not work due to the lack of monotonicity of nonlinearity $M_\theta^{-1}(u)u^p$, and also the variational method in \cite{NS} fails, since (\ref{eq 1.1}) has no variational structure. Furthermore, it is difficult to calculate precise value for  $\int_\Omega |\nabla u_k|$ to express $M_\theta(u_k)$, especially, when $\Omega$ is a general bounded domain.  To overcome these difficulties, we make use of the rearrangement argument to estimate the value of $M_\theta(u)$ and employ the Schauder fixed-point theorem to obtain the existence of isolated singular solutions in the class set of $M_\theta(u)>0$.\smallskip

When $\theta<0$, we can  derive a branch of singular solutions such that   $M_\theta(u)<0$.

 \begin{teo}\label{teo 00}
 $(i)$  Let $N\ge2$, $p\in(1,p^*)$, $\theta<0$ and $\Omega=B_1(0)$.
For $k\in(0,-\theta)$,
 problem (\ref{eq 1.1}) has a nonnegative solution $u_k$, which is a distributional solution of (\ref{eq 1.2}) with $\Omega=B_1(0)$, satisfying that
 $$\theta <M_\theta(u_k)<k+\theta<0$$
  and $u_k$   has the asymptotic behavior
 (\ref{1.2}).

$(ii)$  Let $N\ge2$, $p\in(\frac{N+1}{N-1},p^*)$, $\theta<0$ and $\Omega$ is a bounded smooth domain containing the origin. Then
 problem (\ref{eq 1.1}) has a nonnegative solution $u_p$, which is not a distributional solution of (\ref{eq 1.2}), satisfying that
 $$ M_\theta(u_p) <0$$
  and $u_p$   has the asymptotic behavior
 $$\lim_{|x|\to0^+}u_p(x)|x|^{\frac{2}{p-1}}=c_p(-M_\theta(u_p))^{\frac1{p-1}},$$
 where $c_p=\left[\frac2{p-1}\left(\frac2{p-1}+2-N\right)\right]^{\frac1{p-1}}$.
\end{teo}

For $M_\theta(u)<0$, problem (\ref{eq 1.1}) could be written  as
 \begin{equation}\label{eq 1.3}
 -\Delta u+ \lambda u^p=0\quad {\rm in}\quad {B_1(0)\setminus\{0\}},\quad u=0\quad{\rm on}\quad \partial B_1(0),
 \end{equation}
where $\lambda=-M_\theta^{-1}(u)>0$.
For $\lambda=1$, the nonlinearity in problem (\ref{eq 1.3}) is an  absorption and Lions showed in \cite{Lions1980} that it is always studied  by considering
the very weak solutions of
 \begin{equation}\label{eq 1.4}
 -\Delta u+ \lambda u^p=k\delta_0\quad {\rm in}\quad {B_1(0)}.
 \end{equation}
V\'eron in \cite{V00} gave a survey on the isolated singularities of (\ref{eq 1.3}), in which $B_1(0)$ is replaced by general bounded domain containing the origin.
With a general Radon measure and a more general absorption nonlinearity $g:\R\to\R$  satisfies the subcritical assumption:
$$\int_1^{+\infty}(g(s)-g(-s))s^{-1-p^*}ds<+\infty,$$
problem (\ref{eq 1.4}) has been studied by Benilan-Br\'{e}zis \cite{BB11}, Br\'{e}zis \cite{B12}, by approximating the measure  by a sequence of regular functions, and find classical solutions which
converges to a weak solution. For this approach to work,   uniform bounds for the sequence of classical solutions are necessary to be established. The uniqueness is then derived by
Kato's inequality. Such a method has been applied to solve equations with
boundary measure data  in   \cite{GV,MV1,MV2,MV4} and other extensions  in   \cite{BV,Hung}.

In the case $\lambda=-M_\theta^{-1}(u)$,  depending on the unknown function $u$,  a different approach has  to be  taken  into account   to study  problem (\ref{eq 1.4}). A branch of solutions such that $M_\theta(u)<0$ are derived from the  observations that the function
$F(\lambda)=-M_\theta^{-1}(u)-\lambda$ is continuous and it has a zero,   because we will find two values  $\lambda_1,\lambda_2$ such that $F(\lambda_1)F(\lambda_2)<0$,  where $v_\lambda$ is the unique solution of problem (\ref{eq 1.4}). This zero   indicates a solution of problem (\ref{eq 1.1}).

For the singularity as $|x|^{-2/(p-1)}$,   the diffusion    and  the nonlinear  terms      play the  predominant  roles in (\ref{eq 1.1}), so we just consider $\lambda u_p$, where $u_p$
is the solution of $-\Delta u+u^p=0$ in $\Omega\setminus\{0\}$. By scaling  $\lambda$ to meet the Kirchhoff function and then
a solution with this type singularity is derived  in Theorem \ref{teo 00}.  This scaling technique could be extended to
obtain solutions in the supercritical case in Theorem \ref{teo 2} in Section 5.

It is worth pointing out that the method of searching solutions with the weak singularities as $\Phi$  in Theorem \ref{teo 0} could be extended into dealing with general nonlinearity $f(u)$ when
$0\le f(u)\le c|u|^p$ with $p\in (1,p^*)$.  This method   to prove Theorem \ref{teo 00} is based   on   the homogeneous property of nonlinearity and when the nonlinearity is not a
power function, it is open but challenging to obtain solutions with    such  isolated singularity.

\smallskip

The rest of this paper is organized as follows. In Section 2, we  introduce the very weak solution of equation (\ref{eq 1.1}) involving Dirac mass and give a discussion of (\ref{1.1-1}).   Section 3 is devoted to show the existence  of a solution to   (\ref{eq 1.1}) with $M_\theta(u)>0$ in Theorem \ref{teo 0}. In Section 4, we search the solutions of (\ref{eq 1.1}) with  $M_\theta(u)<0$   in Theorem \ref{teo 00}.   The supercritical case: $N/(N-2)\le p<(N+2)/(N-2)$ with $N\ge3$, is considered in Section 5, and    we obtain there multiple isolated singular  solutions of (\ref{eq 1.1}) such that $M_\theta(u_i)>0$.


\section{ Preliminary}
\subsection{Kirchhoff-type problem with Dirac mass}
In order to drive  solutions of (\ref{eq 1.1}) with  singularity (\ref{1.2}), it is always transformed  into finding solutions of (\ref{eq 1.2}). A function $u$ is said to be a  super (resp. sub) distributional  solution of (\ref{eq 1.2}),
if $u\in L^1(\Omega)$, $|\nabla u|\in L(\Omega)$, $u^p\in L^1(\Omega,\rho dx)$ and
 \begin{equation}\label{e 1.1}
     \int_{\Omega} \left[   u (-\Delta) \xi-\frac{u^p}{ M_\theta(u)}\xi\right]\, d\mu\ge\ ({\rm resp.}\ \le)\ k\xi(0),\quad \forall\, \xi\in C^{1.1}_0(\Omega),\ \xi\ge0,
\end{equation}
where $\rho(x)={\rm dist}(x,\partial\Omega)$. A function $u$ is  a distributional solution  of (\ref{eq 1.2}) if $u$ is both super and sub distributional solutions of of (\ref{eq 1.2}).

Next we build the connection between the singular solutions of (\ref{eq 1.1}) and the distributional solutions of (\ref{eq 1.2}).

 \begin{teo}\label{teo 1}
Assume that $N\ge2$,  $p >1$  and $u\in L^1(\Omega)$ is a nonnegative classical solution of problem (\ref{eq 1.1}) satisfying that  $M_\theta(u)\not=0$ and $u^p\in L^1(\Omega,\rho dx)$.
 Then  $u$ is a very weak solution of problem (\ref{eq 1.2}) for some $k\ge0$. Furthermore,

{\it Case 1: $M_\theta(u)<0$. }\\
 $(i)$  For $N\ge3$, $p\ge p^*$, problem   (\ref{eq 1.1}) only has zero solution and $\theta<0$.   \\
$(ii)$  For $N\ge2$, $1<p< p^*$, we have that $k>0$ and
\begin{equation}\label{2.6}
 \lim_{|x|\to0^*} u(x) \Phi^{-1}(x)=c_Nk.
\end{equation}

{\it Case 2: $M_\theta(u)>0$. }\\
$(i)$  For $N\ge3$, $p\ge p^*$,   we have that $k=0$ and
$$\lim_{|x|\to0}u(x)|x|^{N-2}=0; $$
$(ii)$ Assume more that $1<p<p^*$.
If $k=0$,  then $u$ is removable at the origin,
 and if $k>0$, then $u$ satisfies  (\ref{1.2}).

\end{teo}

In order to prove Theorem \ref{teo 1}, we need the following lemmas.

\begin{lemma}\label{lm 4.1}
Let $\tau\in(0,N)$, then for $x\in B_{1/2}(0)\setminus\{0\}$,
\begin{equation}\label{4.2.6}
\mathbb{G}_\Omega[|\cdot|^{-\tau}](x)\le
\left\{ \arraycolsep=1pt
\begin{array}{lll}
 c_2|x|^{-\tau+2} \quad
 &{\rm if}\quad \tau>2,\\[2mm]
 \displaystyle   -c_2\log (|x|) \quad
 &{\rm if}\quad \tau=2,\\[2mm]
 c_2  \quad
 &{\rm if}\quad \tau<2.
\end{array}
\right.
\end{equation}

For $N\ge 3$, $p\in(1,p^*)$,   there holds
\begin{equation}\label{4.2.1}
\mathbb{G}_\Omega[\mathbb{G}_\Omega^p[\delta_{0}]]\le
\left\{ \arraycolsep=1pt
\begin{array}{lll}
 c_2|x|^{p(2-N)+2} \quad
 &{\rm if}\quad p\in (\frac{2}{N-2},p^*),\\[2mm]
 \displaystyle   -c_2\log (|x|) \quad
 &{\rm if}\quad p=\frac{2}{N-2},\\[2mm]
 c_2  \quad
 &{\rm if}\quad p<\frac{2}{N-2}
\end{array}
\right.
\end{equation}

\end{lemma}
 \begin{proof} We follow the idea of Lemma 2.3 in \cite{CZ}. In fact, from \cite[Propsition 2.1]{BV} it follows that the Green kernel verifies that
$$G_\Omega(x,y)\le c_N \Phi(x-y),$$
By direct computation, we get (\ref{4.2.6}).
Since $\lim_{|x|\to0^+}\mathbb{G}_\Omega[\delta_0](x)\Phi^{-1}(x)\to c_N$,   (\ref{4.2.6}) with $\tau=(2-N)p$ implies (\ref{4.2.1}).
\hfill$\Box$
 \end{proof}

\begin{proposition}\label{embedding}(\cite{St} or \cite[Propostion 5.1]{CFY})
Let $h\in L^s(\Omega)$ with $s\ge1$, then
there exists $c_{3}>0$ such that

\noindent$(i)$
\begin{equation}\label{a 4.1}
\|\mathbb{G}_\Omega[h]\|_{L^\infty(\Omega)}  \le c_{3}\|h\|_{L^s(\Omega)}\quad{\rm if}\quad \frac1s<\frac{2 }N;
\end{equation}

\noindent$(ii)$
\begin{equation}\label{a 4.2}
\|\mathbb{G}_\Omega[h]\|_{L^r(\Omega)}\le c_{3}\|h\|_{L^s(\Omega)}\quad{\rm if}\quad \frac1s\le \frac1r+\frac{2}N\quad
{\rm and}\quad s>1;
\end{equation}

\noindent$(iii)$
\begin{equation}\label{a 4.02}
\|\mathbb{G}_\Omega[h]\|_{L^r(\Omega)}\le c_{3}\|h\|_{L^1(\Omega)}\quad{\rm if}\quad 1<\frac1r+\frac{2}N.
\end{equation}
\end{proposition}

\noindent{\bf Proof of Theorem \ref{teo 1}.}  For $M_\theta(u)\not=0$, we rewrite (\ref{eq 1.1}) as
 \begin{equation}\label{eq 1.1.1}
 \arraycolsep=1pt\left\{
\begin{array}{lll}
 \displaystyle     -\Delta  u =\frac{u^p}{M_\theta(u)}\quad
 & {\rm in}\quad \Omega\setminus \{0\},\\[2mm]
 \phantom{   -\Delta   }
    u=0\quad
  &{\rm on}\quad\partial \Omega.
\end{array}\right.
 \end{equation}

 Since  $u^p\in L^1(\Omega,\rho dx)$ and $u\in L^1(\Omega)$, we may define the operator $L$  by the following
\begin{equation}
L(\xi):=\int_{\Omega} \left[u(-\Delta) \xi  - \frac{u^p}{M_\theta(u)}\xi\right]\,dx,\quad \forall\xi\in C^\infty_c(\R^N).
\end{equation}
First we claim that for any $\xi\in C^\infty_c(\Omega)$ with the support in $\Omega\setminus\{0\}$,
$$L(\xi)=0.$$
In fact, since $\xi\in C^\infty_c(\Omega)$ has the support in $\Omega\setminus\{0\}$, then there exists $r\in(0,1)$ such that $\xi=0$ in $B_r(0)$ and  then
\begin{eqnarray*}
L(\xi) =\int_{\Omega\setminus B_r(0)}\left[u(-\Delta) \xi  - \frac{u^p}{M_\theta(u)}\xi\right]\,dx
  = \int_{\Omega\setminus B_r(0)}\left(-\Delta u  -\frac{u^p}{M_\theta(u)} \right)\xi\,dx
=0.
\end{eqnarray*}
From Theorem 1.1 in \cite{BL},    it implies that
\begin{equation}\label{S}
 L=  k \delta_0\quad {\rm for\ some\ }\ k\ge0,
\end{equation}
 that is,
\begin{equation}
L(\xi)= \int_{\Omega} \left[u(-\Delta ) \xi - \frac{u^p}{M_\theta(u)} \xi\right]\,dx=  k \xi(0),\quad  \quad \forall \xi\in C^\infty_c (\R^N).
\end{equation}
Then $u$ is a weak solution of (\ref{eq 1.2}) for some $k\ge0$.

{\it Case 1: $M_\theta(u)<0$.} We observe that
$$
u=k\mathbb{G}_\Omega[\delta_0] -\frac1{-M_\theta(u)}\mathbb{G}_\Omega[u^q]\le k\mathbb{G}_\Omega[\delta_0],
$$
then
$$
k\mathbb{G}_\Omega [\delta_0]-\frac{k^p}{-M_\theta(u)}\mathbb{G}_\Omega[\mathbb{G}_\Omega [\delta_0]^p]\le u\le k\mathbb{G}_\Omega [\delta_0]\quad{\rm in}\ \ \Omega\setminus\{0\}.
$$
So  if $k=0$, we obtain that $u\equiv 0$, which implies $M_\theta(u)=\theta<0$;
and if $k>0$
$$
\lim_{|x|\to0^+} u(x) \Phi^{-1}(x)=c_Nk.
$$

  We prove that  $k=0$ if $p\ge p^*$ with $N\ge 3$.  By contradiction, if $k>0$, then
 $$u\ge (k/2) \Phi \quad {\rm in}\quad B_{r_0}(0)\setminus\{0\},$$
which implies that
$$u^p(x)\ge  (k/2)^p |x|^{(2-N)p},\quad \forall x\in B_{r_0}(0)\setminus\{0\},$$
where  $(2-N)p\le -N$   and $r_0>0$ is such that $B_{2r_0}(0)\subset \Omega$. A contradiction is obtained that $u^p\not\in L^1(\Omega)$.
Therefore, when $p\ge p^*$, there is no nontrivial nonnegative solution (\ref{eq 1.1}) such that $M_\theta(u)<0$.

 {\it Case 2: $M_\theta(u)>0$.} We refer to \cite{Lions1980} for the proof. For the reader's convenience, we give the details.
 When $p\in (1,N/(N-2))$ and $k=0$,  then
$$u= \frac1{M_\theta(u)}\mathbb{G}_\Omega[u^p ]. $$
We infer from $u^p\in L^{t_0}(\Omega)$ with $t_0=\frac{1}{2}(1+\frac1p\frac{N}{N-2})>1$ and Proposition \ref{embedding} that
$u\in L^{t_1 p}(\Omega)$ and  $u^p\in L^{t_1}(\Omega)$ with
$$t_1=\frac1p\frac{N}{N-2 t_0}t_0>t_0.$$
If $t_1>Np/2$, by Proposition \ref{embedding},
$u\in L^{\infty}(\Omega )$ and then it could be improved that $u$ is a classical solution of
\begin{equation}\label{1.3}
   -\Delta  u =\frac1{ M_\theta(u)}u^p\quad {\rm in}\quad \Omega.
\end{equation}

 If $t_1<Np/2$, we proceed as above.
By Proposition \ref{embedding}, $u\in L^{t_2p}(\Omega)$, where
$$t_2=\frac1p\frac{Nt_1}{N-2 t_1 }>\frac1p\frac{N}{N-2 t_0 }t_1=\left(\frac1p\frac{N}{N-2 t_0}\right)^2t_0.$$
Inductively, let us define
$$t_m=\frac1p\frac{Nt_{m-1}}{N-2 t_{m-1}} >\left(\frac1p\frac{N}{N-2 t_0}\right)^m t_0\to+\infty\quad{\rm as}\quad m\to+\infty.$$
Then there exists $m_0\in\N$ such that
$$t_{m_0}>\frac{1}{2}Np$$
and by part $(i)$ in Proposition \ref{embedding},
$$u\in L^\infty(\Omega).$$
 It then follows that  $u$ is a classical solution of (\ref{1.3}).

 When $p\in (1, N/(N-2))$ and $k\not=0$, we observe that
$$ \lim_{x\to0} \mathbb{G}_\Omega[\delta_0](x)|x|^{N-2} =c_{N,\alpha}$$
and
\begin{equation}\label{13.2.1a}
u=\frac1{M_\theta(u)} \mathbb{G}_\Omega[u^p]+k\mathbb{G}_\Omega[\delta_0].
\end{equation}
 We let
$$u_1= \frac1{M_\theta(u)} \mathbb{G}_\Omega[u^p ]\quad {\rm and}\quad  \Gamma_0=k\mathbb{G}_\Omega[\delta_0].$$
Then by  Young's inequality,
\begin{equation}\label{13.2.1b}
u^p\le 2^p\left(u_1^p+  \Gamma_0^p\right).
\end{equation}
 By the definition of $u_1$ and (\ref{13.2.1b}), we obtain
\begin{equation}\label{13.2.1c}
 u_1 \le 2^p \mathbb{G}_\Omega[u_1^p]+ \Gamma_1,
\end{equation}
where $u_1\in L^s(\Omega)$  for any  $s\in (1,  N/(N-2))$ and
$$ \Gamma_1=2^p \mathbb{G}_\Omega[\Gamma_0^p].$$
Denoting  $\mu_1=2+(2-N)p$, then for $0<|x|< 1/2$,
$$\Gamma_1(x)  \le
\left\{ \arraycolsep=1pt
\begin{array}{lll}
  c_{1}|x|^{\mu_1 } \quad
 &{\rm if}\quad \mu_1<0,\\[2mm]
 \displaystyle   -c_{1}\log |x| \quad
 &{\rm if}\quad \mu_1=0,\\[2mm]
 c_{1} \quad
 &{\rm if}\quad \mu_1>0.
\end{array}
\right. $$
If $\mu_1\le 0$,  letting
$$u_2= 2^p \mathbb{G}_\Omega[ u_1^p], $$
then $u_2\in L^s(\Omega)$ with $s\in[1,\frac{N}{N-2})$, $u_1\le u_2+\Gamma_1$ and
$$u_2\le 2^p\left(\mathbb{G}_\Omega[ u_2^p] + \mathbb{G}_\Omega[ \Gamma_1^p]\right).$$
Let $\mu_2=\mu_1 p +2$, then $\mu_2>\mu_1$ and for $0<|x|<\frac12$,
$$\Gamma_2(x):=2^p \mathbb{G}_\Omega[ \Gamma_1^p](x)\le
\left\{ \arraycolsep=1pt
\begin{array}{lll}
  c_{2}|x|^{\mu_2 } \quad
 &{\rm if}\quad \mu_2<0,\\[2mm]
 \displaystyle   -c_{2}\log |x| \quad
 &{\rm if}\quad \mu_2=0,\\[2mm]
 c_{2}  \quad
 &{\rm if}\quad \mu_2>0.
\end{array}
\right. $$

Inductively, we assume that
$$u_{n-1}\le 2^p \mathbb{G}_\Omega[ u_{n-1}^p]+2^p \mathbb{G}_\Omega[ \Gamma_{n-2}^p], $$
where $u_{n-1}\in L^s(\Omega)$ for $s\in[1,N/(N-2))$, $\Gamma_{n-2}(x)\le |x|^{\mu_{n-2}}$ for  $\mu_{n-2}<0$.

Let
$$u_n= 2^p \mathbb{G}_\Omega[ u_{n-1}^p],\quad \quad \Gamma_{n-1}=2^p\mathbb{G}_\Omega[ \Gamma_{n-2}^p], $$
and
 $$\mu_{n-1}=\mu_{n-2} p +2.$$
Then $u_n\in L^s(\Omega)$ for $s\in[1, N/(N-2))$ and  for $0<|x|<1/2$,
$$\Gamma_{n-1}(x):= \mathbb{G}_\Omega[  \Gamma_{n-2}^p](x)\le
\left\{ \arraycolsep=1pt
\begin{array}{lll}
  c_n|x|^{\mu_{n-1} } \quad
 &{\rm if}\quad \mu_{n-1}<0,\\[2mm]
 \displaystyle   -c_n\log |x| \quad
 &{\rm if}\quad \mu_{n-1}=0,\\[2mm]
 c_n  \quad
 &{\rm if}\quad \mu_{n-1}>0.
\end{array}
\right. $$
We observe that
\begin{eqnarray*}
 \mu_{n-1}-\mu_{n-2}= p(\mu_{n-2}-\mu_{n-3})&=&p^{n-3}(\mu_2-\mu_1)  \\
   &\to&+\infty\quad{\rm as}\quad n\to+\infty.
\end{eqnarray*}
Then there exists $n_2\ge 1$ such that
 $$ \mu_{n_2-1}>0\quad{\rm and}\quad \mu_{n_2-2}\le0$$
and
\begin{equation}\label{3.2.1d1}
u\le u_{n_2}+\sum^{n_2-1}_{i=1} \Gamma_i +\Gamma_0,
\end{equation}
where $\Gamma_i\le c|x|^{\mu_i}$ and
$$u_{n_2}\le 2^p (\mathbb{G}_\Omega[ u_{n_2}^p]+1).$$

 Next, we claim that $u_{n_2}\in L^\infty(\Omega)$. Since $u_{n_2}\in L^s(\Omega)$ for  $s\in[1, N/(N-2))$,
letting
$$t_0=\frac12\left(1+\frac1p\frac{N}{N-2}\right)\in \left(1,\frac{N}{N-2}\right),$$
 then $\frac1p\frac{N}{N-2 t_0 }>1$ and
by Proposition \ref{embedding}, we have that
$u_{n_2}\in L^{t_1}(\Omega)$ with
$$t_1=\frac1p\frac{Nt_0}{N-2 t_0 }.$$
Inductively, it implies by $u_{n_2}\in L^{t_{n-1}}(\Omega)$ that
$u_{n_2}\in L^{t_{n}}(\Omega)$
with
$$t_n=\frac1p\frac{Nt_{n-1}}{N-2 t_{n-1}} >\left(\frac1p\frac{N}{N-2 t_0}\right)^n t_0\to+\infty\quad{\rm as}\quad n\to\infty.$$
Then there exists $n_3\in\N$ such that
$$s_{n_3}>\frac{Np}{2}$$
and by part $(i)$ in Proposition \ref{embedding}, it infers that
$$u_{n_2}\in L^\infty(\Omega).$$

Therefore, it implies by $u\ge \Gamma_0 $ and (\ref{3.2.1d1}) that
$$\lim_{x\to0}u(x)|x|^{N-2}=c_{N,\alpha}k.$$
This ends the proof.
 \hfill$\Box$

\subsection{Discussion on (\ref{1.1-1})  }

The following two functions plays an important role in  searching distributional solutions  of problem (\ref{eq 1.2})
\begin{equation}\label{3.1}
 w_0=\mathbb{G}_\Omega [\delta_0],\qquad w_1=\mathbb{G}_\Omega [w_0^p],
\end{equation}
which are the solutions respectively of
\begin{equation}\label{eq 3.1}
    \arraycolsep=1pt\left\{
\begin{array}{lll}
 \displaystyle    -\Delta  u = \delta_0\quad
 &{\rm in}\quad \Omega,\\[2mm]
 \phantom{   -\Delta    }
 \displaystyle   u=0\quad
 &{\rm on}\quad\partial \Omega
\end{array}\right.
\end{equation}
and
\begin{equation}\label{eq 3.2}
    \arraycolsep=1pt\left\{
\begin{array}{lll}
 \displaystyle    -\Delta  u =w_0^p \quad
 &{\rm in}\quad \Omega,\\[2mm]
 \phantom{    -\Delta    }
 \displaystyle   u=0\quad
 &{\rm on}\quad\partial \Omega.
\end{array}\right.
\end{equation}


Observe that  $a_p>0$ defined in (\ref{1.1-2})  is the smallest constant with $p\in(1,\frac{N}{N-2})$  such that
\begin{equation}\label{2.1}
 w_1\le a_pw_0\quad{\rm in}\ \ \Omega\setminus\{0\}.
\end{equation}
Obviously, $a_p$ depends on the domain $\Omega$.

\begin{proposition}\label{pr 2.1}
Let  $\Omega=B_1(0)$.

$(i)$  If $\theta>0$ and  $1<p<\min\{2,p^*\}$, there exists $a_p^*>0$ depending $\theta$  such that
when $0<a_p\le a_p^*$, (\ref{1.1-1}) holds for any $k>0$; and  when $a_p> a_p^*$, (\ref{1.1-1}) holds
for $0<k\le k_1$ and $k_2\le k<+\infty$, where  $0<k_1<k_2<+\infty$.

  If $\theta>0$, $p^*>2$ and  $2<p< p^*$, there exists $k_3>0$ such that for $0<k\le k_3$, (\ref{1.1-1}) holds.

   If $\theta>0$, $p^*>2$ and  $p=2$, then  when  $ a_2>\frac14,$ (\ref{1.1-1}) holds for $0<k<\frac{\theta}{4a_2-1}$; and when
   $ a_2\le\frac14,$   (\ref{1.1-1}) holds for any $k>0$.\smallskip

$(ii)$  If $\theta=0$ and  $1<p<\min\{2,p^*\}$,  then
  (\ref{1.1-1}) is equivalent to
  $$k\ge  \left(\frac{(p-1)^{p-1}}{p^pa_p}\right)^{-\frac1{2-p}};$$

  If $\theta=0$, $p^*>2$ and  $2<p< p^*$,  then
  (\ref{1.1-1}) is equivalent to
  $$0<k\le  \left(\frac{(p-1)^{p-1}}{p^pa_p}\right)^{\frac1{p-2}}.$$

  If $\theta=0$, $p^*>2$ and  $p=2$, then  when  $ a_2>\frac14,$ there is no $k>0$ such that (\ref{1.1-1}) holds; and when
   $ a_2\le\frac14,$   (\ref{1.1-1}) holds for any $k>0$.\smallskip

$(iii)$  If $\theta<0$ and  $1<p<\min\{2,p^*\}$,  then
  (\ref{1.1-1}) holds for $k\ge k_4$, where
  $$k_4>  \left(\frac{(p-1)^{p-1}}{p^pa_p}\right)^{-\frac1{2-p}};$$

  If $\theta<0$, $p^*>2$ and  $2<p< p^*$,  then   $a_p^{**}=(-\theta)^{2-p}p^{-p}(p-1)(p-2)^{p-3}$  such that
when $0<a_p\le a_p^{**}$, (\ref{1.1-1}) holds for $k_5\le k\le k_6$, where $0<k_5\le\frac{p-1}{2-p}\theta\le  k_6<+\infty$; and  when $a_p> a_p^*$, there is no $k>0$ such that (\ref{1.1-1}) holds.

   If $\theta<0$, $p^*>2$ and  $p=2$, then  when  $ a_2<\frac14,$ (\ref{1.1-1}) holds for $0<k<\frac{\theta}{4a_2-1}$; and when
   $ a_2\ge\frac14,$   (\ref{1.1-1}) holds for any $k>0$.

\end{proposition}
 \begin{proof} When $\Omega=B_1(0)$, we have that $r_0=1$. Let
$$h(k)=\frac{k^{p-1}}{\theta+k}-\frac1{a_pp}\left(\frac{p-1}{p}\right)^{p-1},\quad k\in(\theta_-,+\infty).$$
Note that
$$h'(k)= \frac{(p-1)k^{p-2}(\theta+k)-k^{p-1}}{(\theta+k)^2}.$$
When $p\not=2$, $h'(k_0)=0$ implies that
$$k_0=\frac{p-1}{2-p}\theta.$$

When $p=2$,
$$h(k)=\frac{k}{\theta+k}-\frac1{4a_2},\quad k\in(0,+\infty).$$
 The rest of the proof is simple and hence we omit it.\hfill$\Box$ \medskip
 \end{proof}

When $p=2$, note that $\frac14$ is a critical value for (\ref{1.1-1}) and we show that $a_2<\frac14$ when $\Omega$ is a ball.

\begin{lemma}\label{lm 2.1}
Assume that $\Omega=B_1(0)$, $N=2$ or $3$,  $p=2$ and   $a_2$ is given by  (\ref{2.1}). Then
$$\alpha_2<\frac14.$$
\end{lemma}

 \begin{proof}
  When $\Omega=B_1(0)$, take $\xi(x)=1-|x|$ as a test function, we derive
 \begin{equation}\label{b1}
\int_{B_1(0)}|\nabla w_0|\, dx =  \int_{B_1(0)} \nabla w_0\cdot \nabla(1-|x|) \, dx      =  1.
 \end{equation}

Since $w_1$ is radial symmetric and decreasing, then
$$  -(r^{N-1}w_1'(r))'   =  r^{N-1} w_0^2.$$
So for $N=3$,
$$w_1'(r)=\frac1{16\pi ^2}  r^{-2}\int_0^r(1-t)^2 dt $$
and
$$w_1(r)=\frac1{48\pi ^2}\int_r^1  s^{-2} [(1-s)^3-1]ds=\frac1{48\pi ^2}\left[3(r-1)-3\ln r-\frac{r^2-1}2\right]. $$
Then
$$\frac{w_1(r)}{w_0(r)}= \frac1{12\pi} \left[3r-3\frac{r\ln r}{1-r}+\frac{r+r^2}2\right], $$
then $r\mapsto\frac{w_1(r)}{w_0(r)}$ is increasing,
so
$$a_2=\lim_{r\to1}\frac{w_1(r)}{w_0(r)}=\frac{w_1'(1)}{w_0'(1)}.$$

So for $N=2$,
$$w_1'(r)=\frac1{4\pi ^2}  r^{-1}\int_0^r (\ln t)^2 t\, dt $$
and
$$w_1(r)=\frac1{8\pi ^2}\int_r^1  \left[ s(\ln s)^2 - s\ln s -\frac{s}2\right]ds. $$
Then
$$\frac{w_1(r)}{w_0(r)}= \frac{-\frac{r^2(\ln r)^2}2+ r^2\ln r+\frac{1-r^2}4}{-\frac1{2\pi}\ln r}, $$
then $r\mapsto\frac{w_1(r)}{w_0(r)}$ is increasing,
so
$$a_2=\lim_{r\to1}\frac{w_1(r}{w_0(r)}=\frac{w_1'(1)}{w_0'(1)}.$$

We see that
\begin{eqnarray*}
 -w_1'(1)  = \arraycolsep=1pt\left\{
\begin{array}{lll}
 \displaystyle   \frac1{48\pi^2}   \quad
 &{\rm if}\quad N=3,\\[2mm]
 \phantom{     }
 \displaystyle   \frac1{16\pi^2}    \quad
 &{\rm if}\quad N=2
\end{array}\right.
\end{eqnarray*}
and
$$-w_0'(1)= \arraycolsep=1pt\left\{
\begin{array}{lll}
 \displaystyle   \frac1{4\pi }   \quad
 &{\rm if}\quad N=3,\\[2mm]
 \phantom{     }
 \displaystyle   \frac1{2\pi}    \quad
 &{\rm if}\quad N=2,
\end{array}\right.
 $$
so $$\alpha_2=\arraycolsep=1pt\left\{
\begin{array}{lll}
 \displaystyle   \frac1{12\pi }   \quad
 &{\rm if}\quad N=3,\\[2mm]
 \phantom{     }
 \displaystyle   \frac1{8\pi}    \quad
 &{\rm if}\quad N=2.
\end{array}\right. $$
Therefore, we have that $\alpha_2 <1/4.$
The proof is thus complete.
 \end{proof}

\begin{corollary}\label{cr 0}
 Assume that $N=2$ or $3$, $p=2$ $M_\theta$ is defined by (\ref{kirchhoff}) with $\theta\ge0$, $a_2$ is given by (\ref{1.1-2}), $\Omega=B_1(0)$.
Then for any $k>0$, problem (\ref{eq 1.1}) has a nonnegative solution $u_k$ satisfying (\ref{md}) and (\ref{1.2}).
\end{corollary}


\section{  Solutions with $M_\theta(u)>0$ }

In order to do estimates on $M_\theta(u)$, we introduce   the  following lemma.

\begin{lemma}\label{lm 2.3}
Let $u,v$ be a radially symmetric, decreasing and nonnegative  functions in  $C^1(B_1(0)\setminus\{0\}) \cap W_0^{1,1}(B_1(0))$ such that
\begin{equation}\label{2.2}
 \norm{u}_{L^1(B_1(0))}\ge \norm{v}_{L^1(B_1(0))} \quad{\rm and}\quad \liminf_{|x|\to0^+} [u(x)-v(x)]|x|^{N-1}\ge0.
\end{equation}

Then
$$\int_{B_1(0)}|\nabla u| dx \ge \int_{B_1(0)}|\nabla v| dx. $$

\end{lemma}
\begin{proof} For  radially symmetric decreasing function   $f\in   C^1(B_1(0)\setminus\{0\}) \cap W_0^{1,1}(B_1(0))$,
  we have that
$$\omega_{N} f(r)r^{N-1}+(N-1)\omega_{N} \int_r^1 f(s)s^{N-2} ds =-\omega_{N}\int_r^1f'(s)s^{N-1} ds,$$
then   we have that
$$\omega_{N} \lim_{|x|\to0^+}(u-v)(x)|x|^{N-1}+(N-1)  \int_{B_1(0)} [u(x)-v(x)] dx = \int_{B_1(0)}|\nabla u| dx -\int_{B_1(0)}|\nabla v| dx.$$
From (\ref{2.2}), we have  that
  $$ \int_{B_1(0)}|\nabla u| dx \ge\int_{B_1(0)}|\nabla v| dx.$$
This finishes the proof. \hfill$\Box$
\end{proof}

\smallskip

\noindent{\bf Proof of Theorem \ref{teo 0}.}  We search  for   distributional solutions of
\begin{equation}\label{eq 6.0}
-\Delta u =\frac1{M_\theta(u)}u^p+k\delta_0  \quad{\rm in}\quad  \Omega,\qquad u=0\quad{\rm on}\quad \partial \Omega
\end{equation}
by using the Schauder fixed-point theorem.
 Let $w_0,w_1$ be the solutions of (\ref{3.1}) and  denote
 \begin{equation}\label{3.3}
w_t=tk^pw_1+kw_0,
\end{equation}
where the parameter $t>0$.

{\it  We claim that  there exists   $k_p>0$ independent of $\theta$ such that  for $ k\in(0, k_p]$, if $\theta+ r_0^{-1} k>0$
  there exists $t_p>0$ such that
\begin{equation}\label{6.1}
 t_pk^pw_1\ge \frac{\mathbb{G}_{\Omega}[w_{t_p}^p]}{\theta+ r_0^{-1} k}.
\end{equation}
}

We observe that if
\begin{equation}\label{6.3.5}
\frac{(a_p tk^p+k)^p}{\theta+r_0^{-1}k} \le tk^p  ,
\end{equation}
then $w_t$ verifies (\ref{6.1}), since
\begin{eqnarray*}
\frac{\mathbb{G}_{\Omega}[w_{t_p}^p]}{\theta+ r_0^{-1} k}   &\le &  \frac{(a_p tk^p+k)^p  \mathbb{G}_{\Omega}[w_0^p] }{\theta+ r_0^{-1} k}
\\&=& \frac{(a_p tk^p+k)^p    }{\theta+ r_0^{-1} k} w_1
\\&\ge & tk^p w_1.
\end{eqnarray*}

Now we discuss what condition on $k$  guarantee that (\ref{6.3.5}) holds for some $t>0$. In fact, (\ref{6.3.5}) is equivalent to
\begin{equation}\label{6.3.6}
 (a_p tk^{p-1}+1)^p  \le t  (\theta+ r_0^{-1} k)
\end{equation}
or in the form
$$
s=t  (\theta+ r_0^{-1} k)\quad{\rm and}\quad  \left(\frac{a_p  k^{p-1}}{\theta+ r_0^{-1} k}s +1\right)^p  \le   s.
$$

For $p>1$, since the function $f(s) = (\frac1p(\frac{p-1}p)^{p-1}s+1)^p$ intersects the line $g(s) = s$ at the unique
point $s_p=\left(\frac p{p-1}\right)^{p}$,  so $k$ may be chosen such that
\begin{equation}\label{2.10}
 \frac{a_p  k^{p-1}}{\theta+ r_0^{-1} k}  \le \frac1p\left(\frac{p-1}p\right)^{p-1}.
\end{equation}
In fact, (\ref{1.1-1}) implies \eqref{2.10}.
Therefore, for $k>r_0\theta_-$ satisfying (\ref{1.1-1}) and taking $t_p=(\theta+k)^{-1}\left(\frac p{p-1}\right)^{p}, $
function $w_{t_p}$ verifies (\ref{6.1}).

 Let
\begin{eqnarray*}
\mathcal{D}_k &=&\big\{u\in W^{1.1 }_0(\Omega):  \  0\le u\le  t_pk^pw_1 \big\}.
\end{eqnarray*}
Denote
$$\mathcal{T}u= \frac1{M_\theta (u+kw_0)}\mathbb{G}_\Omega[(u+kw_0)^p],\quad\forall\, u\in \mathcal{D}_k.$$

 {\it We claim that
\begin{equation}\label{2.5}
M_\theta (u+kw_0)\ge \theta+r_0^{-1}k>0\quad{\rm for}\ \ u\in \mathcal{D}_k.
\end{equation}}

For $u\in\mathcal{D}_k$,   we may let $v_n\in C^1_0(\Omega)$  be a sequence of nonnegative functions  converging to  $u$ in $W^{1,1}_0(\Omega)$. Let $u_n=v_n+kw_0$, and  by the fact that $w_0\in C^1(\Omega\setminus\{0\})\cap W_0^{1,1}(\Omega)$, then $u_n\in  C^1(\Omega\setminus\{0\})\cap W_0^{1,1}(\Omega)$, $u_n\ge kw_0$ in $\Omega\setminus\{0\}$  and $u_n$ converge to $u+kw_0$ in $W^{1,1}(\Omega)$.
By the symmetric decreasing arrangement, we may denote  $u^*_n$,   the symmetric decreasing rearranged function of $u_n$ in $B_{r_0}(0)$, where
$r_0\ge1$ such that $|B_{r_0}(0)|=|\Omega|$. Observe that
$$\liminf_{|x|\to0^+} u^*_n(x)|x|^{N-1} \ge 0=k\lim_{|x|\to0^+} w_0(x)|x|^{N-1}$$
and
$$\int_{\Omega} u_n dx\ge k \int_{\Omega} w_0 dx.$$

By P\'{o}lya-Szeg\H{o} inequality, we have that
 $$\norm{\nabla u_n}_{L^1(\Omega)}\ge \norm{\nabla u^*_n}_{L^1(B_{r_0}(0))}=r_0^{-1}\norm{\nabla w^*_n}_{L^1(B_{1}(0))}. $$
where $w_n^*(x)=r_0^{-N}u^*_n(r_0x)$ for $x\in B_1(0)$.

 Let $w_{B_1(0)}=k\mathbb{G}_{B_1(0)}[\delta_0]$, since $B_1(0)\subset \Omega$,  Kato's inequality implies that
$$ \int_{\Omega}w_0 dx \ge  \int_{B_1(0)}w_{B_1(0)} dx.$$
Thus,
\begin{equation}\label{2.4}
\int_{B_1(0)} w^*_n dx=\int_{B_{r_0}(0)} u^*_n dx\ge  \int_{B_1(0)}w_{B_1(0)} dx.
\end{equation}

Thus, by Lemma \ref{lm 2.3}, (\ref{2.4}) and (\ref{b1}), we have
$$\norm{\nabla w_n^*}_{L^1(B_1(0))} \ge \norm{\nabla kw_{B_1(0)}}_{L^1(B_1(0))} =k $$
Therefore, passing to the limit  as $n\to+\infty$ in the above inequality we get that
 $$M_\theta(u+kw_0)\ge \theta+r_0^{-1}\norm{\nabla kw_{B_1(0)}}_{L^1(B_1(0))} =\theta+r_0^{-1}k,$$
which implies (\ref{2.5}).

 Therefore, from (\ref{6.1}) it follows that
 \begin{eqnarray*}
 \mathcal{T}u =  \frac{\mathbb{G}_\Omega[(kw_0+u)^p ]}{M_\theta (kw_0+u)}
\le  \frac{\mathbb{G}_\Omega[(kw_0+t_pk^pw_1)^p ]}{\theta+r_0^{-1}k}
  \le  t_p k^p w_1,
 \end{eqnarray*}
then
$$
  \mathcal{T} \mathcal{D}_k \subset  \mathcal{D}_k.
$$
Note that for $u\in \mathcal{D}_k$, one has that $(u+kw_0)^p\in L^{\sigma}(\Omega)$ with $\sigma\in(1,\frac1p\frac{N}{N-2})$, then  $\mathcal{T} \mathcal{D}_k\subset W^{2,\sigma}(\Omega)$, where   $\sigma\in (1, \frac1p\frac{N}{N-2})$.
Since the embeddings $W^{2,\sigma}(\Omega)\hookrightarrow W^{1,1}(\Omega),\, L^1(\Omega)$  are  compact and
then  $\mathcal{T}$ is a compact operator.

Observing that $\mathcal{D}_k$ is a closed and convex set in $L^1(\Omega)$, we may apply the Schauder
fixed-point theorem to derive that there exists $v_k\in \mathcal{D}_k$ such that
$$\mathcal{T}v_k=v_k. $$
Since $ 0 \le v_k\le  t_pk^pw_1$, so $v_k$ is locally bounded in $\Omega\setminus\{0\}$,
then $u_k:=v_k+kw_0$ satisfies  (\ref{1.2}), and by interior regularity results, $u_k$ is a positive classical solution of (\ref{eq 1.1}).
From Theorem \ref{teo 1}   we deduce that  $u_k$ is a distributional solution of (\ref{eq 1.2}). \hfill$\Box$

\setcounter{equation}{0}
\section{  Solutions with $M_\theta(u)<0$}

For $\theta<0$ and $M_\theta(u)<0$,   equation (\ref{eq 1.2}) could be written as
\begin{equation}\label{eq 3.3}
 -\Delta  u +\frac1{-M_\theta(u)}u^p=k\delta_0  \quad{\rm in}\quad  B_1(0),\qquad u=0\quad{\rm on}\quad \partial B_1(0).
\end{equation}

\begin{lemma}\label{lm sub}

 Let  $p\in(1,p^*)$ and $\lambda>0$. For any $k>0$,   the  problem
\begin{equation}\label{eq 3.3.1}
 -\Delta  u +\lambda u^p=k\delta_0  \quad{\rm in}\quad  B_1(0),\qquad u=0\quad{\rm on}\quad \partial B_1(0)
\end{equation}
has a unique positive weak solution $u_{\lambda,k}$ verifying that
\begin{equation}\label{4.1.0}
 \lim_{|x|\to0^+} u_{\lambda,k}(x)|x|^{N-2}=c_Nk.
\end{equation}
Furthermore,    $u_{\lambda,k}$ is radially symmetric and decreasing with to $|x|$ and
the map $\lambda\mapsto u_{\lambda,k}$ is  decreasing.

\end{lemma}
\begin{proof} The existence could be seen \cite[theorem 3.7]{V} and uniqueness follows by Kato's inequality  \cite[theorem 2.4]{V}.
The radial symmetry of $u_{\lambda,k}$ and decreasing monotonicity with to $|x|$ could be derived by the method of moving plane, see \cite{GNN,S} for the details.
It follows from Kato's inequality that the map $\lambda\mapsto u_{\lambda,k}$ is  decreasing.
The proof ends.\hfill$\Box$\medskip
\end{proof}

\noindent{\bf Proof of Theorem \ref{teo 00}.} $(i)$
 Observe that
$$M_\theta(kw_0)=k\int_{B_1(0)}|\nabla w_0| dx+\theta=k+\theta<0.$$

From Lemma \ref{lm sub} with $\lambda=\lambda_1:=-M_\theta^{-1}(kw_0)$,   problem (\ref{eq 3.3.1}) with $\lambda=\lambda_1$
has a unique solution $v_{\lambda_1}$ verifying
that
$$0<v_{\lambda_1}\le kw_0,$$
then it implies that
$$  k\int_{B_1(0)}|\nabla v_{\lambda_1}| dx\le \int_{B_1(0)}|\nabla kw_0|\, dx$$
and
$$M_\theta(v_{\lambda_1})=k\int_{B_1(0)}|\nabla v_{\lambda_1}| dx+\theta\le k\int_{B_1(0)}|\nabla w_0| dx+\theta=k+\theta,$$
thus,
$$\theta< M_\theta(v_{\lambda_1})<k+\theta, $$
that is,
 \begin{equation}\label{5.1}
 \frac1{-M_\theta(v_{\lambda_1})} <\lambda_1.
 \end{equation}

In terms of Lemma \ref{lm sub}, let $\lambda_2=-M_\theta^{-1}(v_{\lambda_1})$ and $\{v_{\lambda_2}\}$ be the solution of problem (\ref{eq 3.3.1}) with $\lambda=\lambda_2$.
Since $\lambda_2>\lambda_1$, then
$$v_{ \lambda_1}< v_{ \lambda_2}<kw_0.$$
So it follows by Lemma \ref{lm 2.3} that
$$  M_\theta(v_{\lambda_1})< M_\theta(v_{\lambda_2})<  M_\theta(kw_{0}), $$
that is,
 \begin{equation}\label{5.2}
 \frac1{-M_\theta(v_{\lambda_2})} > \lambda_2.
 \end{equation}

\emph{We claim that the map $\lambda\in[\lambda_2,\lambda_1] \mapsto M_\theta(u_{\lambda,k})$ is continuous.}\smallskip

At this moment, we assume that the above argument is true.  Let
$$F(\lambda)=\frac1{-M_\theta(v_{\lambda})}-\lambda,$$
where $v_\lambda$ is the solution of (\ref{eq 3.3.1}) with $\lambda\in[\lambda_2,\lambda_1]$.
  Since   $F$ is continuous in $[\lambda_2,\lambda_1]$,   by (\ref{5.1}), (\ref{5.2})  and the   mean value theorem,
there exists $\lambda_0\in(\lambda_2,\lambda_1)$ such that $F(\lambda_0)=0$,
that is, (\ref{eq 3.3}) has a solution $u_k$ with
$\frac1{-M_\theta(u_k)}=\lambda_0$. From standard regularity, we have that $u_k$ is
a classical solution of (\ref{eq 1.1}) and verifies    the   corresponding properties    in the lemma. \smallskip

Now we prove that the map $\lambda\in[\lambda_2,\lambda_1] \mapsto M_\theta(u_{\lambda,k})$ is continuous. Let $\lambda_2\le \lambda'<\lambda''\le\lambda_1$ and
$u_{\lambda',k}$ and $u_{\lambda'',k}$ be the solutions of (\ref{eq 3.3}) with $\lambda=\lambda'$ and $\lambda=\lambda''$ respectively. Then
$$u_{\lambda'',k}< u_{\lambda',k}$$
and
\begin{equation}\label{5.3}
 M_\theta(u_{\lambda'',k})<M_\theta(u_{\lambda',k}).
\end{equation}

Let $\bar u=u_{\lambda'',k}+(\frac{\lambda''-\lambda'}{\lambda_2})^{1/p}w_0$. Then
\begin{eqnarray*}
 -\Delta \bar u+\lambda'\bar u^p &\ge & -\Delta u_{\lambda'',k}+ \left(\frac{\lambda''-\lambda'}{\lambda_2}\right)^{\frac1p}(-\Delta)w_0 + \lambda'u_{\lambda'',k}^p+\lambda' \frac{\lambda''-\lambda'}{\lambda_2} w_0^p  \\
   &\ge &  -\Delta u_{\lambda'',k}+\lambda'' u_{\lambda'',k}^p
   \\&=&k\delta_0.
\end{eqnarray*}
 Therefore Kato's inequality implies that
$$ u_{\lambda',k}\le u_{\lambda'',k}+\left(\frac{\lambda''-\lambda'}{\lambda_2}\right)^{\frac1p}w_0,$$
which yields that
$$M_\theta(u_{\lambda',k})\le M_\theta(u_{\lambda'',k}) +\left(\frac{\lambda''-\lambda'}{\lambda_2}\right)^{\frac1p}k. $$
  This  together  with (\ref{5.3}),   give
$$|M_\theta(u_{\lambda',k})-M_\theta(u_{\lambda'',k})|\le\left(\frac{\lambda''-\lambda'}{\lambda_2}\right)^{\frac1p}k\to0\quad {\rm as}\ \  |\lambda''-\lambda'|\to0, $$
thus,
the map $\lambda\in[\lambda_2,\lambda_1] \mapsto M_\theta(u_{\lambda,k})$ is continuous.\medskip

$(ii)$   It is    well  known that for  $p\in(1,p^*)$,   the  problem
\begin{equation}\label{eq 3.3.2}
 -\Delta  u + u^p=0 \quad{\rm in}\quad  \Omega\setminus\{0\},\qquad u=0\quad{\rm on}\quad \partial \Omega
\end{equation}
has a   positive solution $v_{p}$ verifying that
\begin{equation}\label{4.1.1}
 \lim_{|x|\to0^+} v_{p}(x)|x|^{\frac{2}{p-1}}=c_p,
\end{equation}
where $c_p=[\frac2{p-1}(\frac2{p-1}+2-N)]^{\frac1{p-1}}$.
Furthermore,    $v_{p}$   is the unique solution of (\ref{eq 3.3.2}) such  that
\begin{equation}\label{4.1.2}
\liminf_{|x|\to0^+} u(x)|x|^{\frac2{p-2}}>0.
\end{equation}

We observe that
$$v_\lambda :=\lambda^{-\frac1{p-1}} v_p$$
is the unique solution of
\begin{equation}\label{eq 3.3.3}
 -\Delta  u +\lambda u^p=0 \quad{\rm in}\quad  \Omega\setminus\{0\},\qquad u=0\quad{\rm on}\quad \partial \Omega
\end{equation}
in the  set of functions satisfying    (\ref{4.1.2}).

 For $p\in (\frac{N+1}{N-1},p^*)$, we have that
$\int_\Omega|\nabla u_p| dx<+\infty$,   so that
$$M_\theta(v_{\lambda}):=\lambda^{-\frac1{p-1}}m_2 +\theta<0\quad{\rm for}\ \ \lambda\in (\lambda_0,+\infty), $$
 where $m_2=\int_{\Omega}|\nabla u_p| dx$ and $\lambda_0=(m_2/(-\theta))^{p-1}$.

  We define
$$F(\lambda):=\frac1{\lambda^{-\frac1{p-1}}m_2 +\theta}+\lambda,\quad \lambda\in(\lambda_0,+\infty).$$
 Observe that  $F$ is continuous, increasing and
  $$\lim_{\lambda\to \lambda_0^+} F(\lambda)=-\infty,\qquad \lim_{\lambda\to +\infty} F(\lambda)=+\infty.$$
  Hence  there exists a unique $\bar\lambda$ such that
$$-\frac1{\bar\lambda^{-\frac1{p-1}}m_2 +\theta}=\bar\lambda.$$
 Meaning that   $-M_\theta^{-1}(v_{\bar\lambda})=\bar \lambda$.
  We then conclude that
(\ref{eq 1.1}) has a solution $u_p:=v_{\bar\lambda}$ with
$ M_\theta(u_p)<0$.  From (\ref{4.1.1}) and     the  definition of $v_\lambda$, we know that $u_p$ is not a weak solution of problem (\ref{eq 1.2}).

\section{ In the supercritical case }

In the super critical case that $ p^*\le p<2^*-1$, we have the following existence results.

\begin{teo}\label{teo 2}
$(i)$ Let $N\ge3$, $ p^*\le p<2^*-1$, $\theta\in\R$ and $\Omega$   be  a bounded smooth domain containing the origin.
 If \\ Case 1: $p>2$,  $p\ge p^*$ and $\theta>0$;\\
Case 2: $p=2\ge  p^*$, $\theta>0$ and $m_2<1$;\\
Case 3: $p^*\le p<2$ and $\theta<0$;\\
Case 4: $p^*\le p<2^*-1$, $p\not=2$ and $\theta=0$,\\
 then
 problem (\ref{eq 1.1}) has two positive solutions $u_{i}$ with $i=1,2$  satisfying that
 $$ M_\theta(u_i) >0,$$
 \begin{equation}\label{7.1}
  {\rm if}\ p\in \left(p^*,\frac{N+2}{N-2}\right),\quad  \lim_{|x|\to0^+}u_i(x)|x|^{\frac{2}{p-1}}= M_\theta(u_i) ^{\frac1{p-1}}c_p
 \end{equation}
 and
\begin{equation}\label{7.2}
{\rm if}\ p=p^*,\quad \lim_{|x|\to0^+}u_i(x)|x|^{N-2}(\ln |x|)^{\frac{N-2}2}= M_\theta(u_i) ^{\frac{N-2}{2}} c_{p^*},
\end{equation}
 where $c_p=[\frac2{p-1}(N-2-\frac2{p-1})]^{\frac1{p-1}}$ and $c_{p^*}=(\frac{N-2}{4})^{N-2}$.

$(ii)$ Let $N=4,\, 5$, $   p=2\in[p^*, 2^*-1)$, $\theta=0$ and $\Omega$   be   a bounded smooth domain containing the origin.
 If  $v$ is a solution of (\ref{eq 6.3.1}) such that $M_\theta(v)=1$, then
 for any $\lambda>0$, $u:=\lambda  v$ is a solution of
 problem (\ref{eq 1.1})  satisfying
 $ M_\theta(u)=\lambda >0$ and (\ref{7.1})--(\ref{7.2}).
\end{teo}

To prove Theorem \ref{teo 2}, we need the following lemma.
\begin{lemma}(\label{lm 6.2}\cite{P1,P2})
 Let $N\ge 3$, $p\in[p^*,\frac{N+2}{N-2})$ and $\Omega$ be a bounded smooth domain containing the origin.     Then the following   problem
\begin{equation}\label{eq 6.3.1}
 -\Delta  u = u^p  \quad{\rm in}\quad  \Omega\setminus\{0\},\qquad u=0\quad{\rm on}\quad \partial \Omega
\end{equation}
has two positive singular solution $v_{1}$ and $v_2$ verifying that
\begin{equation}\label{6.1.0}
{\rm if}\ p\in(p^*,\frac{N+2}{N-2}),\quad \lim_{|x|\to0^+} v_i(x)|x|^{\frac2{p-1}}=c_p
\end{equation}
and
\begin{equation}\label{6.1.1}
{\rm if}\ p=p^*,\quad \lim_{|x|\to0^+} v_i(x)|x|^{N-2}(\ln |x|)^{\frac{N-2}2}=c_{p^*}.
\end{equation}

\end{lemma}

\noindent{\bf Proof of Theorem \ref{teo 2}.}
From Lemma \ref{lm 6.2}, it is known that for  $  p^*\le p<2^*-1$,    problem (\ref{eq 6.3.1}) has
 two positive solutions $v_{i}$ verifying that (\ref{6.1.0}) and (\ref{6.1.1}).

We observe that
$$v_{\lambda,i} =\lambda^{-\frac1{p-1}} v_i$$
is a solution  of
\begin{equation}\label{eq 3.3.4}
 -\Delta  u +\lambda u^p=0 \quad{\rm in}\quad  \Omega\setminus\{0\},\qquad u=0\quad{\rm on}\quad \partial \Omega.
\end{equation}

 For $p^*\le p<2^*-1$,   we have that
$\int_\Omega|\nabla u_p| dx<+\infty$, then
$$M_\theta(v_{\lambda,i})=\lambda^{-\frac1{p-1}}m_i +\theta>0 $$
 for $\lambda\in (0,\lambda_+)$, where $m_i=\int_{\Omega}|\nabla v_i| dx$ and
$$\lambda_+ = \arraycolsep=1pt\left\{
\begin{array}{lll}
 \displaystyle   +\infty  \quad
 &{\rm if}\quad \theta\ge0,\\[2mm]
 \phantom{     }
 \displaystyle   (- m_i/\theta)^{p-1}   \quad
 &{\rm if}\quad \theta<0.
\end{array}\right.
 $$

Denote
$$F_\theta(\lambda)=\frac1{\lambda^{-\frac1{p-1}}m_i +\theta}-\lambda,\quad \lambda\in(0,\lambda_+),$$
which is continuous and
\begin{eqnarray*}
\lim_{\lambda\to \lambda_+ } F_\theta(\lambda) = \arraycolsep=1pt\left\{
\begin{array}{lll}
 \displaystyle   -\infty  \quad
 &{\rm if}\quad \theta>0,\\[2mm]
 \phantom{     }
 \displaystyle   +\infty    \quad
 &{\rm if}\quad \theta<0.
\end{array}\right.
\end{eqnarray*}

Case 1: $p>2$,  $p\ge p^*$ and $\theta>0$,
then there exists $t>0$ such that $$F_\theta(\lambda)>0. $$

Case 2: $p=2\ge  p^*$, $\theta>0$ and $m_i<1$,
then  there exists $t>0$ such that $$F_\theta(\lambda)>0. $$

Case 3: $p^*\le p<2$ and $\theta<0$,
then  there exists $t>0$ such that $$F_\theta(\lambda)<0. $$

In the above three cases, there exists a unique $\bar\lambda_i$ such that
$$\frac1{\bar\lambda_i^{-\frac1{p-1}}m_i +\theta}=\bar\lambda_i,$$
that is,  $M_\theta^{-1}(v_{\bar\lambda_i,i})=\bar \lambda_i$.    Therefore,
(\ref{eq 1.1}) has a solution $u_i:=v_{\bar\lambda_i,i}$ with
$ M_\theta(u_{i})>0$.

When $\theta=0$,
$$F_0(\lambda)= \lambda^{\frac1{p-1}}m_i -\lambda,\quad \lambda\in(0,+\infty),$$
When $N\ge 4$, we have that $1/(p-1)>1$ for $p^*\le p<2^*-1$, or when $N=3$, $p^*\le p<2^*-1$, $p\not=2$,
 $\bar\lambda_i=m_i^{(p-1)/(p-2)}$,
 then
(\ref{eq 1.1}) has a solution $u_i:=v_{\bar\lambda_i,i}$.

 When $N=4,5$ and $p=2\in[p^*,2^*-1)$, if $m_i=1$, then for any $\lambda>0$, $u:=\lambda^{-\frac1{p-1}} v_i$
is a solution (\ref{eq 1.1}) with $M_\theta( u)=\lambda>0$ and verifying (\ref{7.1})--(\ref{7.2}).\hfill$\Box$

\begin{remark}\label{re 2}
{\rm Our method to prove Theorem \ref{teo 2} is based on the homogeneous property of the nonlinearity. When the nonlinearity is not
a power function, this scaling method fails and it is challenging to provide the existence results of isolated singular solutions.}

\end{remark}

\end{document}